\documentclass[12pt,reqno]{amsart}

\usepackage{amssymb}

\usepackage[all]{xy}
\xyoption{arc}

\newcommand{\SL}{{\mathcal{L}}}
\newcommand{\SP}{{\mathcal{P}}}
\newcommand{\SM}{{\mathcal{M}}}
\newcommand{\SO}{{\mathcal{O}}}

\newcommand{\HH}{\mathbb{H}}
\newcommand{\CC}{\mathbb{C}}

\newcommand{\Hom}{\operatorname{Hom}}
\newcommand{\SHom}{{\mathcal{H}om}}
\newcommand{\SEnd}{{\mathcal{E}nd}}
\newcommand{\Ext}{\operatorname{Ext}}

\newcommand{\Spec}{\operatorname{Spec}}
\newcommand{\Pic}{\operatorname{Pic}}

\newcommand{\too}{\longrightarrow}
\newcommand{\rk}{\operatorname{rk}}
\newcommand{\id}{\operatorname{id}}

\newtheorem{proposition}{Proposition}[section]
\newtheorem{theorem}[proposition]{Theorem}
\newtheorem{lemma}[proposition]{Lemma}
\newtheorem{definition}[proposition]{Definition}

\theoremstyle{definition}
\newtheorem{remark}[proposition]{Remark}

\numberwithin{equation}{section}

\begin{document}

\baselineskip=15pt

\title[Comparison of Poisson structures on moduli spaces]{Comparison of Poisson
structures on moduli spaces}

\author[I. Biswas]{Indranil Biswas}

\address{School of Mathematics, Tata Institute of Fundamental
Research, Homi Bhabha Road, Mumbai 400005, India}

\email{indranil@math.tifr.res.in}

\author[F. Bottacin]{Francesco Bottacin}

\address{Dipartimento di Matematica, Universit\`a degli Studi di Padova,
Via Trieste 63, 35121 Padova, Italy}

\email{bottacin@math.unipd.it}

\author[T. L. G\'omez]{Tom\'as L. G\'omez}

\address{Instituto de Ciencias Matem\'aticas (CSIC-UAM-UC3M-UCM),
Nicol\'as Cabrera 15, Campus Cantoblanco UAM, 28049 Madrid, Spain}

\email{tomas.gomez@icmat.es}



\date{}

\begin{abstract}
Let $X$ be a complex irreducible smooth projective curve,
and let ${\mathbb L}$ be
an algebraic line bundle on $X$ with a nonzero section $\sigma_0$.
Let $\SM$ denote the moduli space of stable Hitchin pairs
$(E,\, \theta)$, where $E$ is an algebraic vector bundle on $X$ of fixed rank $r$ and degree $\delta$,
and $\theta\, \in\, H^0(X,\, \SEnd(E)\otimes K_X\otimes{\mathbb L})$. Associating
to every stable Hitchin pair its spectral data, an isomorphism of $\SM$ with a moduli space $\SP$ of
stable sheaves of pure dimension one on the total space of $K_X\otimes{\mathbb L}$
is obtained. Both the moduli spaces $\SP$
and $\SM$ are equipped with algebraic Poisson structures, which are constructed
using $\sigma_0$. Here we prove that the above isomorphism
between $\SP$ and $\SM$ preserves the Poisson structures.
\end{abstract}

\maketitle

\section{Introduction}

Let $X$ be a complex irreducible smooth projective curve of genus $g$. Take an
algebraic line bundle $N$ on $X$ such that $N\otimes K^{-1}_X$ admits a nonzero section,
where $K_X$ is the canonical bundle of $X$; fix a section $\sigma_0\,\in\, H^0(X,\, N\otimes K^{-1}_X)
\setminus\{0\}$. Fix integers $r$ and $\delta$. Let
$\SM$ denote the moduli space of stable pairs 
of the form $(E,\,\theta)$, where $E$ is an
algebraic vector bundle on $X$ of rank $r$ and degree $\delta$, and $\theta\, \in\, H^0(X,\, \SEnd(E)
\otimes N)$. These are called Hitchin pairs; when $N\,=\, K_X$,
they are called Higgs bundles. The moduli space
$\SM$ is nonempty if one of the following four assumptions hold
(cf.~\cite[Remark 3.4]{Ma}):
\begin{enumerate}
\item the genus $g$ of $X$ is at least $2$;

\item $g \,=\, 1$, $N \,=\, K_X$ and $\gcd(r,\,\delta) \,= \,1$;

\item $g \,= \,1$ and $\deg(N \otimes K_X^{-1}) \,>\, 0$; and

\item $g = 0$ and $\deg(N \otimes K_X^{-1}) \,\ge\, 3$.
\end{enumerate}
We assume that one of these four conditions hold.
This moduli space $\SM$ has
a natural algebraic Poisson structure \cite{Bo1}, \cite{Ma}, which is constructed using
$\sigma_0$.

In the special case where $N\,=\, K_X$ and $\sigma_0$ is the constant function $1$,
this $\SM$ is a moduli space of Higgs bundles. In that case, the Poisson structure is nondegenerate,
meaning it is a symplectic structure; this symplectic structure was constructed earlier
in \cite{Hi1}, \cite{Hi2}. Furthermore, there is a natural algebraic $1$--form on $\SM$
such that the symplectic form is the exterior derivative of it.

Let $S$ denote the smooth complex quasi-projective surface defined by the total space of the line bundle 
$N$. Given any Hitchin pair $(E,\,\theta)\,\in\, \SM$, there
is a subscheme $Y_{(E,\theta)} \, \subset\, S$ of dimension one associated to it. Furthermore,
associated to the pair $(E,\,\theta)$ there is a coherent sheaf ${\SL}_{(E,\theta)}$
on $Y_{(E,\theta)}$ which is pure of dimension one.
This pair $(Y_{(E,\theta)},\, {\SL}_{(E,\theta)})$ is known as the spectral datum
associated to $(E,\,\theta)$. This construction produces an algebraic map
\begin{equation}\label{i1}
\Phi\, :\, \SM\, \longrightarrow\, \SP\, ,
\end{equation}
where $\SP$ is a moduli space of stable sheaves of pure dimension one on $S$. This construction is
reversible, and $\Phi$ is in fact an isomorphism (see \cite{Hi1}, \cite{Hi2}, \cite{Si},
\cite{BNR}).

The earlier mentioned section $\sigma_0$ of $N\otimes K^{-1}_X$ produces a Poisson structure on
the surface $S$. An algebraic Poisson structure on $\SP$ is constructed using this
Poisson structure on $S$ \cite{Bo2}, \cite{Ty}.

It may be mentioned that when $N\,=\, K_X$ and $\sigma_0$
is the constant function $1$, the Poisson structure on $S$ is the canonical
symplectic form on the cotangent bundle $K_X$ of $X$.
In that case the Poisson structure on $\SP$ coincides
with the symplectic structure on $\SP$ constructed by Mukai in \cite{Mu}.
Also, in this case there is a natural algebraic $1$--form on $\SP$
such that the symplectic form is the exterior derivative of it.

In this continuation of the papers \cite{Bo1,Bo2} of the second author, our aim is to prove the
following (see Theorem \ref{thm1}):

\begin{theorem}\label{thm0}
The isomorphism $\Phi$ in \eqref{i1} takes the Poisson structure on $\SM$ to the
Poisson structure on $\SP$.
\end{theorem}

When $N\,=\, K_X$, related results can be found in \cite{HK}, \cite{HH}, \cite{BM}; there the symplectic
form on a moduli space of Higgs bundles is compared with the symplectic form on the Hilbert scheme of
zero dimensional subschemes of fixed length of the total space of $K_X$ (Mukai had shown that this
Hilbert scheme has a symplectic structure).

\section{Hitchin pairs and spectral data}

Let $X$ be an irreducible smooth complex projective curve of genus $g$, and let $N$ be a fixed
algebraic line bundle on $X$. A Hitchin pair $(E,\,\theta)$ is an algebraic vector bundle $E$ on $X$
together with a morphism $\theta\,:\,E\,\too \,E\otimes N$ of $\SO_X$-modules \cite{Hi1,Ni}. 

A Hitchin pair $(E,\,\theta)$ is \emph{stable} (respectively,
\emph{semistable}) if
$$
\frac{\deg F}{\rk F} \,< \,\frac{\deg E}{\rk E} \ \ \ \text{(respectively,}\ \frac{\deg F}{\rk F} \,
\leq\,\frac{\deg E}{\rk E}\text{)}
$$
for all subbundles $0\, \not=\, F\,\subsetneq\, E$ for which $\theta(F)\, \subset\, F\otimes N$.

Let $\overline\SM$ denote the moduli space of all S-equivalence classes of semistable Hitchin pairs
(constructed in \cite{Ni}; see \cite[Definition 4.2]{Ni} for definition of S-equivalence) of fixed rank $r$ and degree $\delta$. Let 
$$
\SM\,\subset\, \overline\SM
$$
be the open subset of stable Hitchin pairs. We remark that the S-equivalence class of a stable Hitchin pair is the same thing as its isomorphism class, so this moduli space parametrizes isomorphism classes of objects.

\begin{lemma}
\label{smoothness}
The moduli space $\SM$ of stable Hitchin pairs 
$(E,\theta: E\to E\otimes N)$, with $\deg(E)=\delta$, $\rk(E)=r$,
and $N = K_X(D)$ for an effective divisor $D\geq 0$, is smooth.
\end{lemma}
\begin{proof}
This is well-known, but since we did not find a reference, we include
a sketch of the proof.
Fix a line bundle $\Delta\in \Pic^\delta(X)$ on $X$ of degree $\delta$,
and let $\SM_{\Delta}$ be the moduli space of stable Hitchin pairs
with fixed determinant $\det(E)\cong \Delta$ and ${\rm tr}(\theta)=0$.
Consider the determinant morphism
\begin{equation}
  \label{eq:det}
\det:\SM \longrightarrow \Pic^\delta(X) \times H^0(N)
\end{equation}
sending $(E,\theta)$ to $(\det(E),{\rm tr}(\theta))$.
This morphism is surjective, and the target is
smooth, so to prove smoothness of $\SM$ it is enough to show that the
morphism is flat and the fibers are smooth.
Note that the fiber over
a point $(L,\alpha)$ is isomorphic to the moduli space $\SM_{\Delta}$. Indeed, 
let $\eta\in J(X)$ such that $\eta^r=L\otimes \Delta^{-1}$, 
and send $(E,\theta)\in \SM_{\Delta}$ to $(E\otimes \eta,
\theta + \frac{1}{r}\alpha)$. We can choose a root $\eta$
in an \'etale neighbourhood of $\Delta$, so the determinant
morphism is \'etale locally trivial, and in particular it is
flat. Therefore, it only remains to prove that $\SM_\Delta$ is
smooth. 

The deformation theory of the moduli space $\SM$ of Hitchin pairs of fixed
degree is governed by the complex (in degrees 0 and 1)
$$
C^{\bullet}_{\SEnd}:\SEnd(E) \stackrel{[\cdot,\theta]}\too \SEnd(E)\otimes N
$$
This is proved in \cite{BR}. In particular, 
the infinitesimal deformations are given by $\HH^1(C^{\bullet}_{\SEnd})$,
and $\HH^2(C^{\bullet}_{\SEnd})$ is the obstruction space. Furthermore,
the exact sequence of complexes
$$
0 \too \SEnd(E)\otimes N[-1] \too C^{\bullet}_{\SEnd} \too \SEnd(E) \too 0
$$
($[-1]$ means that the sheaf is placed in degree 1) 
gives an exact sequence
\begin{equation}
  \label{eq:h0end}
0  \too \HH^0(C^{\bullet}_{\SEnd}) \too H^0(\SEnd(E))
\stackrel{[\cdot,\theta]}\too 
H^0(\SEnd(E)\otimes N)
\end{equation}
and then $\HH^0(C^{\bullet}_{\SEnd})$ represents the endomorphisms of
$(E,\theta)$ in the category of Hitchin pairs, i.e., endomorphisms
of $E$ which commute with $\theta$.

On the other hand, the deformation theory of the moduli space
$\SM_\Delta$ of stable Hitchin pairs of fixed determinant is given
by the complex
$$
C^{\bullet}_{ad}:ad(E) \stackrel{[\cdot,\theta]}\too ad(E)\otimes N
$$
This can be checked by a straightforward modification of the proofs 
in \cite{BR}, requiring at
each step that the deformation of $E$ does not change its determinant,
and that the deformation of $\theta$ preserves the vanishing of its
trace.

It is well-known that if $E$ is a stable vector bundle, then its only
endomorphisms are given by scalar multiplication
(cf.~\cite[Proposition 1.2.7 and Corollary 1.2.8]{HL}).
A straightforward modification of the usual argument for vector bundles
shows that this is also true in the category of Hitchin pairs, i.e., 
if $(E,\theta)$ is stable as a Hitchin pair, and $\varphi:E\to E$
is a homomorphism which commutes with $\theta$, then $\varphi$
is the multiplication by a scalar. Therefore, using \eqref{eq:h0end} we have
$$
\HH^0(C^{\bullet}_{\SEnd})=\CC \cdot \id_E
$$
Using the exact sequence
$$
0  \too \HH^0(C^{\bullet}_{ad}) \too H^0(ad(E))
\stackrel{[\cdot,\theta]}\too 
H^0(ad(E)\otimes N)
$$
we obtain
\begin{equation}
  \label{eq:inf}
  \HH^0(C^{\bullet}_{ad})=0
\end{equation}
because this cohomology is given by the
endomorphisms of $E$ which are traceless and commute with $\theta$.

We claim that $\HH^2(C^{\bullet}_{ad})$ vanishes. Indeed, this group is dual to
\begin{equation}
  \label{eq:comp}
\HH^0(ad(E)\otimes N^{-1}\otimes K_X \too ad(E)\otimes K_X)  
\end{equation}
If $N\cong K_X$, then this is just $\HH^0(C^{\bullet}_{ad})$, which we have just seen
that it is zero \eqref{eq:inf}.
On the other hand, if $N\cong K_X(D)$ with $D>0$ an effective, nonzero divisor, then 
note that the complex which appears in \eqref{eq:comp} 
is just $C^{\bullet}_{ad}\otimes \SO_X(-D)$.
We have
an exact sequence of complexes
$$
0 \too 
C^{\bullet}_{ad}\otimes \SO_X(-D)
\too 
C^{\bullet}_{ad} \too C^{\bullet}_{ad}\otimes \SO_D \too 0
$$
The associated long exact cohomology sequence starts as
$$
0 \too \HH^0(C^{\bullet}_{ad}\otimes \SO_X(-D))
\too \HH^0(C^{\bullet}_{ad})
\too \HH^0(C^{\bullet}_{ad}|_D)
$$
The vanishing of $\HH^0(C^{\bullet}_{ad})$ \eqref{eq:inf} implies the vanishing
of $\HH^0(C^{\bullet}_{ad}\otimes \SO_X(-D))$. So we have proved that
the obstruction space $\HH^2(C^{\bullet}_{ad})$ vanishes.

Summing up, there are no infinitesimal isomorphisms, and the
obstruction space vanishes, so $\SM_\Delta$ is smooth, and then the
smoothness of $\Pic^\delta(X)$ implies the smoothness of $\SM$.
\end{proof}

In the sequel we assume that one of the following four assumptions hold:
\begin{enumerate}
\item $g \,\ge\, 2$;

\item $g \,=\, 1$, $N\, =\, K_X$ and $\gcd(r,\,\delta)\, =\, 1$;

\item $g \,=\, 1$ and $\deg(N \otimes K_X^{-1}) \,>\, 0$;

\item $g \,=\, 0$ and $\deg(N \otimes K_X^{-1}) \,\ge\, 3$.
\end{enumerate}
This assumption implies the existence of stable Hitchin pairs \cite[Remark 3.4]{Ma}, and hence $\SM$ is nonempty.

We will now recall the \emph{spectral construction}, which is a bijective correspondence
between Hitchin
pairs $(E,\,\theta)$ and certain sheaves $\SL$ on the total space
\begin{equation}\label{et}
p\, :\, S\, :=\,\mathbb{V}(N)\, \longrightarrow\, X
\end{equation}
of
the line bundle $N$. 
This construction is in \cite[\S~5]{Hi2} for smooth spectral
curves, \cite[p.~173--174, Proposition 3.6]{BNR} for integral spectral curves, 
and \cite[p.~18, Lemma 6.8]{Si} in general (see also \cite[p.~173--174, Remark 3.7]{BNR} and \cite{Sc}).

Consider the projection $p$ in \eqref{et}. Let $x$ denote the
tautological section of $p^*N$ on $S\,=\,\mathbb{V}(N)$.
For a Hitchin pair $(E,\,E\stackrel{\theta}{\too} E\otimes N)$, we define the
associated sheaf $\SL$ using the following short exact sequence of coherent sheaves on $S$
\begin{equation}\label{eq:resolution}
0 \,\too\, p^*(E\otimes N^\vee) \,\stackrel{h}{\too}\, p^* E \,\too\, \SL \,\too\, 0\, ,
\end{equation}
where $p$ is the projection in \eqref{et} and $h\, :=\,p^*\theta - x$, with $x$
being the tensor product with the above mentioned tautological section $x$; the homomorphism
$E\otimes N^\vee\, \too\, E$ given by $\theta$ is denoted here by $\theta$ also.
Throughout, the dual will be denoted by the superscript ``$\vee$''.
The homomorphism $h$ in \eqref{eq:resolution} is
injective (it is an isomorphism over the generic point of $S$),
so \eqref{eq:resolution} is indeed a short exact sequence of coherent sheaves. The
spectral curve $$Y\,\subset\, S$$
for $(E,\, \theta)$ is the subscheme defined by the characteristic polynomial $$\det(p^* \theta-x)
\,=\,0\, .$$ Equivalently, the spectral curve $Y$ is the subscheme defined by the
$0$--th Fitting ideal sheaf $${\rm Fitt}_0(\SL)\,\subset\,
\SO_S\, .$$ To see this equivalence, note that \eqref{eq:resolution} gives a presentation
of $\SL$, and, by definition, the $0$--th Fitting ideal is the ideal
generated by the codimension zero minors of the morphism $h$ in \eqref{eq:resolution}. Since the
ranks of the source and target of $h$ are equal, we only have one minor which
is the determinant $\det(h)\,=\,\det(p^* \theta-x)$, in other words the characteristic
polynomial of $\theta$. We have the following diagram
\begin{equation}\label{ei}
\xymatrix{
{Y}\ar[r]^{i} \ar[rd]_{\pi}& {S}\ar[d]^{p} \\
& {X}
}
\end{equation}
The spectral curve $Y$ is given as the zero of $\det(h)$, which is a section
of $p^*N^r$, and it is proper over $X$, of degree $r$. 
Note that $\SL$ is pure of dimension 1 because $p_*\SL=E$ is torsion free. 
The Euler characteristic of $\SL$ is $\chi(\SL)=\chi(E)=\deg(E)+r(1-g)$. 
We collect all these properties in the following definition.

\begin{definition}
\label{spectralsheaf}
A spectral sheaf $\SL$ on $S$, with invariants $r$ and $\delta$, 
is a coherent sheaf such that
\begin{enumerate}
\item $\SL$ is pure of dimension 1
\item The subscheme $Y\subset S$ defined by its 0-th Fitting ideal is in the divisor class given by $p^*N^r$
\item The projection $Y\longrightarrow X$ is proper, finite of degree $r$.
\item $\chi(\SL)=\delta +r(1-g)$
\end{enumerate}
\end{definition}

If the spectral curve $Y\subset S$ associated to a spectral sheaf $\SL$ is smooth, then there is a line bundle $L$ on $Y$ such that 
$\SL\,=\,i_* L$, where $i$ is the map in \eqref{ei} (\cite[p.~173--174, Remarks 3.7 and 3.8]{BNR}).

Conversely, let $\SL$ be a spectral sheaf on $S$.
The push-forward 
\begin{equation}\label{di}
E\,=\,p_*\SL
\end{equation}
is then a torsion free coherent sheaf, of rank $r$ and degree $\delta$. 
Define the Hitchin
pair $(E,\, \theta)$ 
by setting
$\theta\,=\,p_*x$, where $p$ is the projection in \eqref{ei}. More explicitly,
consider the multiplication by the tautological section $x$ of $p^* N$
$$
\SL\,\stackrel{x}{\too}\, \SL \otimes p^* N
$$
and then $\theta\,=\,p_* x$ is defined as the push-forward of this
homomorphism, using the projection formula
$$
\theta\,:\,E \,=\, p_*\SL\,\too \, p_*(\SL\otimes p^* N)\,=\, (p_*\SL)\otimes N
\,=\,E\otimes N
$$
(see \eqref{di}).
The above reversible construction gives a correspondence between the 
Hitchin pairs $(E,\,\theta)$ 
and the spectral sheaves $\SL$ on $S$

Note that the push-forward $p_*\SL$ is coherent if and only if the projection $Y \longrightarrow X$ is proper, and this holds if and only if the closure of $Y$ in $\overline{S}\,=\,\mathbb{P}(N\oplus \SO_X)$ lies in $S$ (\cite[p.~18, Lemma 6.8]{Si}).

We shall now recall the definition of stability given by Simpson  (\cite{Si})
for a spectral sheaf $\SL$.
Let $$P(\SL,m)\,=\,\chi(\SL\otimes p^*\SO_X(m))$$
be its Hilbert polynomial. Its degree $d$ is equal to the dimension of the
support of $\SL$, and the leading coefficient is 
$a m^d/d!$, where $a\,=\,a(\SL)$ is an integer which we call the rank of
$\SL$. We say that $\SL$ is \emph{stable} (respectively,
\emph{semistable}) if for all proper subsheaves $0\, \not=\, \SL'\,\subset\, \SL$,
$$
a(\SL) P(\SL',m) \,<\, a(\SL') P(\SL,m)\ \ \ \text{ (respectively, }\,\ a(\SL) P(\SL',m)
\,\leq\, a(\SL')P(\SL,m)\text{)}
$$
It is easy to see that if $\SL$ is semistable, then it has pure
dimension, i.e., the dimension of the support of every nonzero subsheaf is the
same as the dimension of the support of $\SL$. 
Note that, in the spectral construction, the condition that $\SL'$ is
a subsheaf of $\SL$ translates into the condition that the
subbundle 
$$
E'\,:=\,p_*\SL'\,\subset\,E\,:=\,p_*\SL
$$ 
satisfies the condition $\theta(E')\, \subset\, E'\otimes N$. It follows that a Hitchin pair is
stable (respectively, semistable) if and only if the corresponding
spectral sheaf on $S$ is
stable (respectively, semistable) \cite[p.~19, Corollary 6.9]{Si}.

The spectral construction can be carried out for families, so we get an
isomorphism between the moduli functors. Since the stability
conditions coincide, we get an isomorphism between 
the moduli space of stable Hitchin pairs $\SM$
or rank $r$ and degree $\delta$ and the moduli space $\SP$ of stable
spectral sheaves with invariants $r$ and $\delta$
\begin{equation}\label{ep}
\Phi\,:\,\SM \,\too\, \SP\, .
\end{equation}
In particular, both moduli spaces, $\SM$ and $\SP$, are smooth 
(Lemma \ref{smoothness}). 

We are interested in calculating the differential of the
isomorphism $\Phi$ in \eqref{ep}. More generally, we will give an isomorphism between the
infinitesimal deformation space of a Hitchin pair $(E,\,\theta)$ 
and the infinitesimal deformation space of the 
corresponding sheaf $\SL$ on $S$.

\section{Poisson structure on the moduli spaces}

The infinitesimal deformation space of a spectral sheaf (Definition \ref{spectralsheaf}) $\SL$ is given by
\begin{equation}\label{tanm}
\Ext^1(\SL,\,\SL)
\end{equation}
\cite[p.~425, (3.1)]{Bo2}. To calculate the dual of this vector space we will use Serre
duality. Recall that $\SL$ is a coherent sheaf on the surface $S\,=\,\mathbb{V}(N)$ in
\eqref{et} which is not projective. However, it has a projective compactification 
$$
j\,:\,S \,\too\, \overline{S}\,=\,\mathbb{P}(N\oplus \SO_X) \, ;
$$
furthermore, the direct image $j_*\SL$ is a coherent sheaf on $\overline{S}$, because the 
closure, in $\overline{S}$, of the support of $\SL$ does not meet 
the boundary $\overline{S}\setminus S$ (this argument is in 
\cite[p.~18, Lemma 6.8]{Si}). 
We have
\begin{equation}\label{dh1}
\Ext^1(\SL,\,\SL)^\vee\,=\,\Ext^1_{\SO_{\overline{S}}}(j_*\SL,\,j_*\SL)^\vee
\,\cong\, \Ext^1_{\SO_{\overline{S}}}(j_*\SL,\,j_*\SL\otimes K_{\overline{S}})
\,=\,\Ext^1(\SL,\,\SL\otimes K_S)\, ,
\end{equation}
where the second isomorphism is Serre duality on the projective
surface $\overline{S}$. 

Let $\SP$ be the moduli space of stable spectral sheaves on $S$ (cf. \eqref{ep}). The tangent space 
$T_{[\SL]}\SP$ at the point corresponding to $\SL$ is canonically identified with $\Ext^1(\SL,\,\SL)$, and the 
cotangent space $T^\vee_{[\SL]}\SP$ is canonically identified with $\Ext^1(\SL,\,\SL\otimes K_S)$ (see
\eqref{dh1}).

Henceforth assume that the line bundle $N\otimes K^{-1}_X$ on $X$ admits a nonzero
section. We fix a nonzero section
\begin{equation}\label{es0}
\sigma_0\, \in \, H^0(X,\, N\otimes K^{-1}_X)\setminus \{0\}\, .
\end{equation}
Since the anticanonical line bundle $K_S^\vee$ of the surface $S$ in
\eqref{et} is identified with $p^* (N\otimes K^{-1}_X)$, where $p$ is
the projection in \eqref{et}, we have
\begin{equation}\label{ess}
s\, :=\, p^* \sigma_0\, \in\, H^0(S,\, K_S^\vee)\, ,
\end{equation}
where $\sigma_0$ is the section in \eqref{es0}. This $s$ produces 
an algebraic Poisson structure on $\SP$
\begin{equation}\label{fp}
B\,:\,T^\vee \SP \,\too\, T\SP
\end{equation}
\cite{Bo2}, \cite{Ty}. We will recall below the fiberwise construction of $B$.

Tensoring with the section $s$ produces a homomorphism $\SL\otimes K_S\, \longrightarrow\,
\SL$. Consequently, we obtain a homomorphism
\begin{equation}\label{poissonsheaf}
T^\vee_{[\SL]} \SP\,\cong\, \Ext^1(\SL,\,\SL\otimes K_S)\,\too\, \Ext^1(\SL,\,\SL)\,\cong\, 
T_{[\SL]}\SP\, .
\end{equation}
The isomorphism $\Ext^1(\SL,\,\SL)\,\cong\, T_{[\SL]}\SP$ is a consequence of the fact that
the infinitesimal deformations of $\SL$ are parametrized by $\Ext^1(\SL,\,\SL)$, while the
isomorphism $T^\vee_{[\SL]} \SP\,\cong\, \Ext^1(\SL,\,\SL\otimes K_S)$ follows from \eqref{dh1}.
Let $B([\SL])$ be the restriction to the fiber over the point $[\SL]$ of
the bundle map $B$ in \eqref{fp}.
This homomorphism $B([\SL])$ coincides with the
one in \eqref{poissonsheaf} (see \cite[p.~428, Formula (4.4)]{Bo2}).

On the other hand, the infinitesimal deformation space of a
Hitchin pair $(E,\, \theta)$ is given by the first hypercohomology of a complex
\begin{equation}\label{hd}
\HH^1( 
E^\vee \otimes E \,\stackrel{[\cdot,\,\theta]}{\too}\, E^\vee \otimes E \otimes N )
\end{equation}
(see \cite[p.~399, Proposition 3.1.2]{Bo1}, \cite[p.~271, Proposition 7.1]{Ma}, \cite[p.~220, 
Theorem 2.3]{BR}). The dual vector space
$$
\HH^1(
E^\vee \otimes E \,\stackrel{[\cdot,\,\theta]}{\too}\, E^\vee \otimes E \otimes N )^\vee
$$
is calculated using Serre duality for hypercohomologies. Denote
by $A^\bullet$ the complex 
$$
A^\bullet\,:\, \quad E^\vee \otimes E \,\stackrel{[\cdot,\,\theta]}{\too}\, E^\vee \otimes E \otimes N 
$$
concentrated in degrees $0$ and $1$. 
Serre duality gives an isomorphism
$$
\HH^1 (A^\bullet)^\vee \stackrel{\cong}{\too} \HH^{-1}(A^\bullet{}^\vee[1] \otimes K_X) = 
\HH^1(A^\bullet{}^\vee[-1]\otimes K_X)
$$
\cite[p.~67, Theorem 3.12]{Hu}.
The dual complex 
$$
A^\bullet{}^\vee\,: \quad E \otimes E^\vee \otimes N^\vee \,\stackrel{[\cdot,\,\theta]^{t}}{\too}
\, E \otimes E^\vee 
$$
is concentrated in degrees $-1$ and $0$.
We identify $E \otimes E^\vee$ and $E^\vee \otimes E$ by switching the
factors; then the morphism $[\cdot,\,\theta]^{t}$ becomes
$[\theta,\,\cdot]$, so
\begin{equation}\label{es}
A^\bullet{}^\vee\,: \quad E^\vee \otimes E \otimes N^\vee \,\stackrel{[\theta,\,\cdot]}{\too}
\, E^\vee \otimes E\, .
\end{equation}
Finally we shift the complex in \eqref{es} by $[-1]$. Recall that, when we shift a complex by
one unit, we multiply by $-1$ all the differentials 
\cite[p.~28, Definition 2.4]{Hu}, 
so we get the complex, concentrated in degrees $0$ and $1$
\begin{equation}\label{es2}
A^\bullet{}^\vee[-1]\otimes K_X\,: \quad E^\vee \otimes E \otimes N^\vee\otimes K_X
\,\stackrel{[\cdot,\,\theta]}{\too} \,E^\vee \otimes E\otimes K_X\, .
\end{equation}

\begin{remark}
In \cite[p.~402, Proposition 3.1.10]{Bo1} the morphism is actually
$[\theta,\,\cdot]$ instead of the homomorphism $[\cdot,\, \theta]$ in \eqref{hd}. Of course
there is an isomorphism between the two complexes (multiplication by $-1$
in one term and $+1$ in the other). 
\end{remark}

As in \eqref{ep}, let $\SM$ be a moduli space of stable Hitchin pairs such that the rank and degree
of the vector bundle underlying a Hitchin pair are $r$ and $\delta$ respectively.
Using the section $\sigma_0$ in \eqref{es0} an algebraic Poisson structure
\begin{equation}\label{sp}
B^H\, :\, T^\vee \SM \,\too\, T\SM
\end{equation}
on $\SM$ is constructed \cite[\S~4.3]{Bo1}, \cite[p.~278, Corollary 7.15]{Ma}.
We recall below a fiberwise construction of the homomorphism $B^H$ in \eqref{sp}.

Tensoring with
the section $\sigma_0$ in \eqref{es0} induces homomorphisms
$N^\vee \otimes K_X\, \longrightarrow\, {\SO}_X$ and $K_X\, \longrightarrow\, N$.
Using these we obtain a homomorphism of complexes
\begin{equation}\label{eq:morcom}
\begin{matrix}
{E^\vee \otimes E \otimes N^\vee \otimes K_X} & \stackrel{[\cdot,\,\theta]}{\longrightarrow}
& {E^\vee \otimes E \otimes K_X}\\
\sigma_0 \Big\downarrow\,\,\,\,\,\,\, && \,\,\,\,\,\,\,\Big\downarrow \sigma_0 \\ 
{E^\vee \otimes E} & \stackrel{[\cdot,\,\theta]}{\longrightarrow} &
{E^\vee \otimes E\otimes N}.
\end{matrix}
\end{equation}
The morphism of complexes in \eqref{eq:morcom} gives a morphism of cohomologies
\begin{align}
\HH^1( 
E^\vee \otimes E \otimes N^\vee \otimes K_X & \,\stackrel{[\cdot,\,\theta]
}{\too}\, E^\vee \otimes E \otimes K_X ) \nonumber \\
\too\,\HH^1( 
E^\vee \otimes E & \stackrel{[\cdot,\,\theta]}\too E^\vee \otimes E
\otimes N)\, .
\label{poissonhiggs}
\end{align}
We have $T_{[(E,\theta)]}\SM\,=\, \HH^1(
E^\vee \otimes E \stackrel{[\cdot,\,\theta]}\too E^\vee \otimes E\otimes N)$, because the
infinitesimal deformations of $(E,\, \theta)$ are parametrized by
$\HH^1(E^\vee \otimes E \stackrel{[\cdot,\,\theta]}\too E^\vee \otimes E\otimes N)$ (see \eqref{hd});
this and \eqref{es2} together imply that
\begin{equation}\label{zl}
T^\vee_{[(E,\theta)]}\SM\,=\, \HH^1( 
E^\vee \otimes E \otimes N^\vee \otimes K_X \,\stackrel{[\cdot,\,\theta]
}{\too}\, E^\vee \otimes E \otimes K_X)\, .
\end{equation}
Using \eqref{zl} and \eqref{hd}, the homomorphism in \eqref{poissonhiggs} becomes a homomorphism
$$
T^\vee_{[(E,\theta)]} \SM \,\too\, T_{[(E,\theta)]}\SM\, .
$$
This homomorphism coincides with $B^H([(E,\theta)])$ in \eqref{sp}
\cite[\S~4.3]{Bo1}, \cite[p.~278, Corollary 7.15]{Ma}.

\begin{remark}
If we compare \eqref{eq:morcom} with the commutative diagram in
\cite[Remark 1.3.3]{Bo1}, we see that the sign is changed in the
vertical left arrow and in one of the horizontal arrows
(there is $[\theta,\,\cdot]$ in \cite[Remark 1.3.3]{Bo1} as
opposed to $[\cdot,\,\theta]$ here).
\end{remark}

\section{Comparison of Poisson structures}

We shall compare the Poisson structures $B$ and $B^H$, constructed in \eqref{ep} and
\eqref{sp} respectively, using the isomorphism $\Phi$ in \eqref{ep}.

We first
recall that the global Ext can be calculated using locally free
resolutions and hypercohomology.

\begin{lemma}[{\cite[Corollary 2 to Theorem 4.2.1]{Gr}}]\label{lem:isom}
Let $V$ and $W$ be coherent sheaves on a scheme $T$. Let
$$
\cdots\,\too\, L^{-3} \,\too\, L^{-2} \,\too\, L^{-1} \,\too\, L^{0} \,\too\, V \,\too\, 0
$$
be a resolution of $V$ by finitely generated locally free sheaves. Then there is an
isomorphism
$$
\Ext^i(V,\,W) \,\cong\, \HH^i(\SHom(L^\bullet,\,W))
$$
which is functorial on $W$.
\end{lemma}

\begin{proposition}\label{prop:com}
Let $W$ be a coherent sheaf on the surface $S$ in \eqref{et}. Let $x\,:\,W\,\too\, W\otimes p^*N$
be multiplication by the tautological
section of $p^* N$ on $S$, where $p$ is the projection in \eqref{et}. Then
for any $(E,\, \theta)\, \in\, \SM$ (see \eqref{ep}) there is an isomorphism
$$
\varphi_W\,\,\,:\,\,\,
\Ext^1(\SL,\,W) \,\stackrel{\cong}{\too}\, 
\HH^1(E^\vee \otimes p_*W \,\stackrel{f_W}{\too}\, E^\vee 
\otimes N\otimes p_*W)\, ,
$$
where $\SL\, =\, \Phi((E,\, \theta))$ (see \eqref{ep}),
$f_W\,=\,\theta^\vee \otimes 1 - 1 \otimes p_*x$, and $f_W$ is seen as
a complex concentrated in degrees $0$ and $1$.
This isomorphism is functorial on $W$, meaning a homomorphism $V\,\too\, W$ of coherent sheaves
on $S$ induces a commutative diagram
$$
\xymatrix{
{\Ext^1(\SL,\,V)} \ar[r] \ar[d] & 
{\HH^1( 
E^\vee \otimes p_*V \stackrel{f_V}\too E^\vee 
\otimes N\otimes p_*V
)} \ar[d] \\
{\Ext^1(\SL,\,W)} \ar[r]& {\HH^1( 
E^\vee \otimes p_*W \stackrel{f_W}\too E^\vee 
\otimes N \otimes p_*W
)}
}
$$
\end{proposition}

\begin{proof}
The short exact sequence in \eqref{eq:resolution} gives a locally free
resolution $L^\bullet$
of $\SL$ (concentrated on degrees $-1$ and $0$)
$$
L^\bullet\,:\quad p^*(E\otimes N^\vee) \,\stackrel{h}{\too}\, p^* E\,.
$$
For any sheaf $W$ on $S$, Lemma \ref{lem:isom} gives an isomorphism functorial for $W$
\begin{equation}\label{eq:iso1}
\Ext^1(\SL,\,W)\,=\,\HH^1(\SHom(L^\bullet,\,W))\, .
\end{equation}
Consider the diagram
$$
\xymatrix{{S} \ar[d]_{f} \ar[r]^{p} & {X} \ar[dl]^{\gamma}\\
{\Spec \CC}
} 
$$
where $f$ and $\gamma$ are the structure morphisms.
Note that the hypercohomology of any complex $D^\bullet$ on $S$
can be calculated as the cohomology of the complex $Rf_*D^\bullet$
(where $Rf_*$ is the derived functor between the derived categories).
We have 
$$
Rf_*\,=\, R(\gamma_* \circ p_*)\,\stackrel{\cong}{\too}\,
R\gamma_* \circ Rp_*\, .
$$
The morphism $p$ is affine, hence $p_*$ is exact and then
$Rp_*\,=\,p_*$. If we apply this observation to the complex
$\SHom(L^\bullet,W)$, we get
\begin{equation}\label{eq:isop}
\HH^1(\SHom(L^\bullet,\,W)) \,\stackrel{\cong}{\too} 
\,\HH^1(p_* \SHom(L^\bullet,\, W))\, .
\end{equation}
In view of \eqref{eq:iso1} and \eqref{eq:isop}, to finish the proof, 
it only remains to show that the complex
$p_*\SHom(L^\bullet,W)$ is equal to the complex $f_W$ in the statement
of the proposition. 

Applying $\SHom(\cdot,\,W)$ to the complex $L^\bullet$ we obtain
a complex concentrated in degrees $0$ and $1$
\begin{equation}\label{c1}
\SHom(L^\bullet,\,W)\,:\quad p^* E^\vee \otimes W \,\stackrel{h'}\too\, p^* E^\vee \otimes p^* N \otimes W\, ,
\end{equation}
where $h'\, :=\, p^* \theta^\vee \otimes 1 - 1\otimes x$. By the
projection formula, $p_*\SHom(L^\bullet,W)$ is
$$
E^\vee \otimes p_*W\,=\, p_*(p^* E^\vee \otimes W)
\,\stackrel{p_*h'}{\too}\, p_*(p^* E^\vee \otimes p^* N \otimes W)\,=\, E^\vee \otimes N\otimes p_*W\, ,
$$
where $p_*h'\,=\,\theta^\vee \otimes 1 - 1\otimes p_*x$, and the proposition follows.
\end{proof}

\begin{remark}
\label{rmk1}
In the proof of Proposition \ref{prop:com}, the main step is the isomorphism in
\eqref{eq:isop}, because it facilitates the passage of objects defined on $S$ to objects
defined on $X$. Note that this isomorphism is induced by the push-forward $p_*$ using $p$.
\end{remark}

Let
\begin{equation}\label{ed}
d\Phi\, :\, T\SM\, \longrightarrow\, \Phi^*T\SP
\end{equation}
be the differential of the isomorphism $\Phi$ in \eqref{ep}.

\begin{lemma}\label{commtan}
Let
$$
\varphi\,\,\,:\,\,\,
\Ext^1(\SL,\,\SL) \,\stackrel{\cong}{\too}\, 
\HH^1(E^\vee \otimes E \,\stackrel{f_{\SL}}{\too}\, E^\vee 
\otimes N\otimes E)
$$
be the homomorphism given by Proposition \ref{prop:com}
for $W\,=\,\SL$ (see \eqref{di}). Then the following diagram
is commutative
$$
\begin{matrix}
\Ext^1(\SL,\,\SL) & \stackrel{\varphi}{\longrightarrow} &
{\HH^1(E^\vee \otimes E \stackrel{[\cdot,\,\theta]}\too E^\vee \otimes E
\otimes N )}\\
\,\,\,\,\Big\uparrow\alpha && \,\,\,\,\Big\downarrow\beta\\
T_{[\SL]}\SP& \xleftarrow{d\Phi(\Phi^{-1}([\SL]))} &
T_{[(E,\theta)]}\SM
\end{matrix}
$$
where $\alpha$ and $\beta$ are the infinitesimal deformation 
maps in \eqref{tanm} and \eqref{hd} respectively, and
$d\Phi(\Phi^{-1}([\SL]))$ is the homomorphism in \eqref{ed}
at $\Phi^{-1}([\SL])\, \in\, \SM$.
\end{lemma}

\begin{proof}
Consider an infinitesimal deformation $\SL_\epsilon$ of $\SL$, i.e.,
a sheaf on $S \times \Spec(\mathbb{C}[\epsilon]/(\epsilon^2))$
flat over $\Spec(\mathbb{C}[\epsilon]/(\epsilon^2))$,
that fits in a short exact sequence
\begin{equation}\label{infdef}
0 \,\too\, \SL \,\stackrel{\epsilon}{\too}\, \SL_\epsilon \,\stackrel{{\rm mod}\,\epsilon}{\too}\,
\SL \,\too\, 0\, .
\end{equation}
By applying the functor $\Hom(\SL,\, \cdot)$ to this short exact sequence we obtain
a long exact sequence
$$
\cdots\, \too\, \Hom(\SL,\,\SL_\epsilon) \,\too\, \Hom(\SL,\,\SL) \,\stackrel{\delta}{\too}\,
\Ext^1(\SL,\,\SL) \,\too\, \cdots\, ;
$$
the element of $\Ext^1(\SL,\, \SL)$ that corresponds to the infinitesimal
deformation $\SL_\epsilon$ is $\delta(\id)$, the image of the identity
$\id \,\in\, \Hom(\SL,\, \SL)$ by the connecting homomorphism $\delta$.

If we set $W\,=\,\SL$ in Proposition \ref{prop:com}, we have
$p_*W\,=\,E$ and $p_*x\,=\,\theta$; then 
$f_W\,=\,[\cdot,\,\theta]$, and Proposition \ref{prop:com} produces
the isomorphism $\varphi$. 

The element
\begin{equation}\label{tv}
\beta\circ\varphi\circ\alpha(\SL_\epsilon)\, \in\,
T_{[(E,\theta)]}\SM\, ,
\end{equation}
where $\SL_\epsilon\, \in\, T_{[\SL]}\SP$ is the element in
\eqref{infdef}, has the following description. 

Consider the short exact sequence \eqref{infdef}. 
Applying $p_*$ to it, we obtain an exact sequence of complexes
\begin{equation}\label{u1}
0 \,\too\, K^\bullet \,\stackrel{\epsilon}{\too}\, K^\bullet_\epsilon
\,\stackrel{{\rm mod}\,\epsilon}{\too}\, K^\bullet \,\too\, 0\, ,
\end{equation}
where $K^\bullet$ denotes the complex
$0 \,\too\, E \,\stackrel{\theta}{\too}\, E \otimes N \,\too\, 0$. Since
$K^\bullet_\epsilon$ in \eqref{u1} is a family of Hitchin pairs on $X$ parametrized by
$\Spec(\mathbb{C}[\epsilon]/(\epsilon^2))$, it corresponds to an element of
$T_{[(E,\theta)]}\SM$. This element of $T_{[(E,\theta)]}\SM$
coincides with the one in \eqref{tv}. Now from the construction of the map
$\Phi$ it follows that
$$
d\Phi(\Phi^{-1}([\SL]))(\beta\circ\varphi\circ\alpha(\SL_\epsilon))\, \in\,
T_{[\SL]}\SP
$$
coincides with $\SL_\epsilon\,\in\, T_{[\SL]}\SP$.
\end{proof}

We remark that the key to the commutativity of the diagram in the statement of
Lemma \ref{commtan} is that both $\varphi$ and $\Phi^{-1}$ are induced by the
push-forward $p_*$.

\begin{lemma}
\label{commcotan}
Consider
$$
\varphi_1\, :=\,(\varphi_{\SL\otimes K_S})^{-1}
\, :\,{\HH^1(E^\vee \otimes E \otimes K_X\otimes N^\vee
\stackrel{[\cdot,\,\theta]}\too E^\vee \otimes E \otimes K_X )}
\,\longrightarrow\,\Ext^1(\SL,\,\SL\otimes K_S),
$$
where $\varphi_{\SL\otimes K_S}$ is the isomorphism in Proposition
\ref{prop:com} for $W\,=\,\SL\otimes K_S$. Let $(d\Phi)^\vee([\SL])$ be dual of the
differential of the isomorphism $\Phi$ at $[\SL]$. Then the following diagram is commutative
$$
\begin{matrix}
\Ext^1(\SL,\,\SL\otimes K_S) & \stackrel{\varphi_1}{\longleftarrow} &
{\HH^1(
E^\vee \otimes E \otimes K_X\otimes N^\vee
\stackrel{[\cdot,\,\theta]}\too
E^\vee \otimes E \otimes K_X )}\\
\,\,\,\,\Big\downarrow\alpha^\vee && \,\,\,\,\Big\uparrow\beta^\vee\\
{T^\vee_{[\SL]}\SP} & \xrightarrow{(d\Phi)^\vee([\SL])} &
T^\vee_{[(E,\theta)]}\SM
\end{matrix}
$$
where $\alpha^\vee$ and $\beta^\vee$ are the natural isomorphisms obtained from
\eqref{dh1} and \eqref{es2} respectively, and $(d\Phi)^\vee$ is the dual of the
homomorphism in \eqref{ed}.
\end{lemma}

\begin{proof}
Consider the homomorphism $\varphi$ in Lemma \ref{commtan}.
For any $$\omega\, \in\, {\HH^1(E^\vee \otimes E \otimes K_X\otimes N^\vee
\stackrel{[\cdot,\,\theta]}\too E^\vee \otimes E \otimes K_X )}$$
and $v\, \in\, \Ext^1(\SL,\,\SL)$, we have
\begin{equation}\label{zl2}
\varphi_1(\omega)(v)\, =\, \omega(\varphi(v))\, \in\, {\mathbb C}\, ;
\end{equation}
recall that
$$
\HH^1(E^\vee \otimes E \otimes K_X\otimes N^\vee
\stackrel{[\cdot,\,\theta]}\too E^\vee \otimes E \otimes K_X )\,=\,
\HH^1(E^\vee \otimes E \,\stackrel{[\cdot,\,\theta]}\too\, E^\vee
\otimes E\otimes N)^\vee
$$
(see \eqref{zl} and \eqref{hd}), and 
$$
\Ext^1(\SL,\,\SL\otimes K_S)\,=\, \Ext^1(\SL,\,\SL)^\vee
$$
(see \eqref{dh1}). From \eqref{zl2} it follows immediately that
$\varphi_1$ coincides with the dual homomorphism $\varphi^\vee$.

Therefore, every homomorphism in the diagram in this lemma is the dual of the
corresponding homomorphism in the diagram in Lemma \ref{commtan}. Hence the
lemma follows from Lemma \ref{commtan}.
\end{proof}

The homomorphism $B$ in \eqref{fp} gives an algebraic section
\begin{equation}\label{b1}
B\, \in\, H^0(\SP,\, \bigwedge\nolimits^2 T\SP)\, ,
\end{equation}
and the homomorphism $B^H$ in \eqref{sp} gives an algebraic section
$$
B^H\, \in\, H^0(\SM,\, \bigwedge\nolimits^2 T\SM)\, .
$$
Consider the homomorphism $d\Phi$ in \eqref{ed}. We note that
\begin{equation}\label{b2}
\Phi_* B^H\, :=\, (\bigwedge\nolimits^2 d\Phi) (B^H)\, \in\, H^0(\SP,\, \bigwedge\nolimits^2 T\SP)
\end{equation}
is a Poisson structure on $\SP$.

\begin{theorem}\label{thm1}
The isomorphism $\Phi$ in \eqref{ep} satisfies the condition $$\Phi_* B^H\,=\, B\, ,$$
where $\Phi_* B^H$ and $B$ are the sections in \eqref{b2} and \eqref{b1} respectively.
\end{theorem}

\begin{proof}
The section $\sigma_0\,\in\, H^0(X,\, N\otimes K^\vee_X)$ gives a section
$s$ of $K_S^\vee\,=\,p^*(N\otimes K^\vee_X)$, and hence it produces a homomorphism
$\SL\otimes K_S\,\too\, \SL$. We apply Proposition \ref{prop:com} to this homomorphism and
get the following commutative diagram:
\begin{equation}\label{cd}
\xymatrix{
{\Ext^1(\SL,\SL\otimes K_S)} \ar[r]^-{\cong}_-{\varphi'} \ar[d] &
{\HH^1( 
E^\vee \otimes E \otimes K_X\otimes N^\vee 
\stackrel{[\cdot,\,\theta]}\too
E^\vee \otimes E \otimes K_X )} \ar[d]
\\
{\Ext^1(\SL,\,\SL)} \ar[r]^-{\cong}_-{\varphi} &
{\HH^1( 
E^\vee \otimes E \,\stackrel{[\cdot,\,\theta]}{\too}\, E^\vee \otimes E
\otimes N )}
}
\end{equation}
where the left vertical homomorphism is the one in \eqref{poissonsheaf} and the right
vertical homomorphism is the one in \eqref{poissonhiggs}. 

Lemma \ref{commcotan} shows that the top horizontal homomorphism $\varphi'$ in \eqref{cd} 
coincides with $(d\Phi)^\vee ([\SL])$. We note that $\varphi'\,=\, (\varphi_1)^{-1}$, where 
$\varphi_1$ is the homomorphism in Lemma \ref{commcotan}. On the other hand, Lemma 
\ref{commtan} shows that the bottom horizontal homomorphism in \eqref{cd} coincides with 
$d\Phi^{-1}([\SL])$. The left vertical homomorphism in \eqref{cd} gives $B([\SL])$, the Poisson 
structure on $\SP$, while the right vertical homomorphism in \eqref{cd} gives 
$B^H([(E,\theta)])\,=\, B^H(\Phi^{-1}([\SL]))$, the Poisson structure on $\SM$. Consequently, 
the theorem follows from the commutativity of the diagram in \eqref{cd}.
\end{proof}

\section*{Acknowledgements}

We thank Oscar Garc{\'\i}a-Prada for helpful discussions on the
smoothness and deformation theory of the moduli space of Hitchin
pairs. We thank that referee for comments which improved the exposition.
The first author is supported by a J. C. Bose fellowship. The third author is supported
by Ministerio de Ciencia e Innovaci\'on of Spain (grants PID2019-108936GB-C21 and
ICMAT Severo Ochoa project CEX2019-000904-S) and CSIC (2019AEP151 and \textit{Ayuda 
extraordinaria a Centros de Excelencia Severo Ochoa} 20205CEX001).

\end{document}